\documentclass{siamart220329}

\usepackage{amssymb,amsmath,amsfonts,amscd,verbatim,stmaryrd}
\usepackage{color,graphicx}
\usepackage{algorithm,algorithmic}
\usepackage{upgreek,comment,url}
\usepackage{mystylefile}
\usepackage{float}
\usepackage{color}
\newcommand{\bmat}[1]{\begin{bmatrix}#1 \end{bmatrix}}

\title{Stable Rank and Intrinsic Dimension\\ of 
Real and Complex Matrices\thanks{The work of the first author was supported in part by NSF grants DMS-1745654,
DMS-1760374 and CCF-2209510, and DOE grant DE-SC0022085. The work of the second author was supported in part by NSF grants DMS-1845406 and DMS-1745654.}
}
\author{Ilse C.F.\ Ipsen\thanks{Department of Mathematics, 
North Carolina State University, Raleigh, NC 27695-8205, USA,
\email{ipsen@ncsu.edu}} \and Arvind K.\ Saibaba\thanks{Department of Mathematics, North Carolina State University, Raleigh, NC 27695-8205, USA, \email{asaibab@ncsu.edu}}}
\headers{Stable rank and intrinsic dimension}{I.C.F. Ipsen and A.K. Saibaba}
\begin{document}
\maketitle

\begin{abstract}
\Ilse{The notion of `stable rank' of a matrix is central to the analysis of randomized matrix algorithms, covariance estimation, deep neural networks, and recommender systems.}
We compare the properties of the stable rank of a real or complex matrix, and the related
concept of 'intrinsic dimension' of a Hermitian positive semi-definite matrix to those of the classical rank.
Basic proofs and examples illustrate that the stable rank does not satisfy any of the 
 fundamental rank properties, while the intrinsic dimension satisfies
 a few.
 In particular, the stable rank and intrinsic dimension of
 a submatrix can exceed those of the original matrix;  adding a Hermitian positive semi-definite matrix can lower the intrinsic dimension 
 of the sum; and multiplication by a nonsingular matrix can
 drastically change the stable rank and the intrinsic dimension.
 We generalize the concept of stable rank to the $p$-stable rank in a Schatten $p$-norm, thereby unifying the concepts of stable rank and intrinsic dimension: The stable rank is the 2-stable rank,
 while the intrinsic dimension is the 1-stable rank of a 
 Hermitian positive semi-definite matrix.
 We derive sum and  product inequalities for the 
 $p$th root of the $p$-stable rank, and show that it is well-conditioned in the norm-wise absolute sense.
 The conditioning improves if the
 matrix and the perturbation are Hermitian positive semi-definite.
\end{abstract}



\section{Introduction}
The \textit{rank} of a real or complex matrix $\ma\in\cmn$ is equal to the 
number of non-zero singular values.
\Ilse{The rank determines the dimensions of the four fundamental subspaces of
$\ma$, and can be
viewed as quantifying the amount of information 
inherent in the matrix.}

\Ilse{However, the rank is ill-posed 
because it does not depend in a Lipschitz continuous manner on the elements
of the matrix. A tiny perturbation
in $\ma$ can change its rank. Rank deficiency can cause
difficulties in least squares/regression problems, bifurcation methods, and neural networks \cite{Stewart1984,Stewart1987,FZH2022}.}
\Ilse{\paragraph{Numerical rank} 
To circumvent the 
illposedness of the traditional rank,
one often resorts to the concept of `numerical rank'
\cite[Section 5.4.1]{GovL13}, which is the rank obtained
after setting a certain number of small singular values
equal to zero. The default computation of the numerical rank in 
Matlab\footnote{\url{https://www.mathworks.com/help/matlab/ref/rank.html}}
is the number of singular values larger than 
$\max\{m,n\}\mathtt{eps}(\|\ma\|_2)$, where the second factor
denotes the distance of $\|\ma\|_2$ to the next larger floating point number.}

\Ilse{Estimating the numerical rank requires effort.
For instance, a randomized algorithm for 
estimating the numerical rank of a $m\times n$ matrix with 
$m\geq n$ and desired numerical rank~$r$ requires 
$\mathcal{O}(mn\log{n}+r^3)$ arithmetic operations
\cite{MN24}.
}

\Ilse{\subsubsection*{Stable rank}
In contrast to the rank,}
the \textit{stable rank of a 
matrix}\footnote{\Ilse{Not to be confused with other concepts of stable rank, such as the stable rank of
an algebra \cite{Kawamura24} or the G-stable rank of a tensor \cite{Derksen22}.}}
\cite{rudelson2007sampling,Tropp2015,Versh18}
\begin{align*}
sr(\ma)\equiv \|\ma\|_F^2/\|\ma\|_2^2
\end{align*}
\Ilse{can be
viewed as `a continuous counterpart of the rank'
\cite[Section 1]{GFMA21}.}
According to \cite[Section 7.8]{Versh18}, the concept of 
`stable rank of a matrix', also called `effective rank' or `numerical rank', was introduced in~\cite[Remark 1.3]{rudelson2007sampling}
and `is a robust version of the classical, linear, 
algebraic rank' \cite[Section 7.6.1]{Versh18}. 
\Ilse{Therefore, the stable rank can serve as a quantitatively rigorous and efficient surrogate for the previously discussed
numerical rank \cite{LWZ2024}.}
\Ilse{In addition to robustness to small perturbations, the stable rank also has the advantages of differentiability, of being upper 
bounded
by the rank, and scaling invariance \cite[section 1]{STD20}.}

The stable rank `tends to be 
low when $\ma$ is close to a low rank matrix' 
\cite[Remark 1.3]{rudelson2007sampling}.
For instance, a matrix $\ma\in\cnn$ with singular values
$1/2^{j-1}$, $1\leq j\leq n$, has
\begin{align*}
\rank(\ma)=n\qquad \text{but}\qquad 
\sr(\ma)=\tfrac{4}{3}\left(1-\tfrac{1}{n}\right)\leq 
\tfrac{4}{3}.
\end{align*}


\Ilse{\subsubsection*{Intrinsic dimension} A}
related concept for Hermitian, or real symmetric, positive semi-definite matrices $\ma\in\cnn$
is that of \textit{intrinsic dimension} of a matrix
\cite{HKZ2012,Tropp2015,Versh18}
\begin{align*}
\intdim(\ma)\equiv \trace(\ma)/\|\ma\|_2.
\end{align*}

The stable rank and intrinsic dimension have the advantage
of being `stable
under small perturbations' \cite[Remark 1.3]{rudelson2007sampling}. Consequently, the stable
rank and the intrinsic dimension are well-posed,
in these sense of Hadamard \cite[Section 1.3]{Isaacson66}.

\subsubsection*{Applications}
Both, the stable rank and the intrinsic dimension appear in the random matrix theory literature, \Ilse{where they replace the explicit dependence of the bounds on the dimensions of the random matrices.} Examples include the matrix Bernstein and matrix Chernoff inequalities~\cite[Chapter 7]{Tropp2015}, and the analysis of covariance estimation~\cite{Vershynin2009,Versh18}.  
\Ilse{Below we discuss the importance of the stable 
rank in the 
analysis of randomized algorithms, covariance estimation, Deep Neural Networks, and recommender systems.}

\Ilse{\paragraph{Randomized algorithms} 
The concepts of stable rank and intrinsic dimension are crucial in the analysis of randomized matrix algorithms,
such as matrix multiplication~\cite{rudelson2007sampling}, trace estimation~\cite{roosta2015improved}, diagonal estimation~\cite{hallman2023monte}, and low-rank approximations~\cite{halko2011finding}. }

\Ilse{
For example, consider the Monte Carlo approximation of the Gram matrix $\ma\ma^T$ for a matrix $\ma \in \rmn$ with more columns than rows, $n\gg m$. Let
\begin{align*}
\ms = \frac{1}{\sqrt{N\pi_{t_j}}} \begin{bmatrix} \ve_{t_1} & \cdots & \ve_{t_N}\end{bmatrix}
\end{align*}
be a sampling matrix that selects $N$ columns with indices $\{t_1,\ldots,t_N\}$ uniformly, independently and with replacement from $\{1, \ldots, n\}$ according to the probabilities $\pi_j = \|\ma\ve_j \|_2^2/ \|\ma\|_F^2$, $1 \le j \le n$.
With the help of a matrix Bernstein concentration inequality,
one can bound the error due to randomization of the 
Monte Carlo approximation $(\ma\ms)(\ma\ms)^T$ as in \cite[Theorem 4.2]{HoI15}.
Given $ 0 < \epsilon \le 1$ and $0 < \delta < 1$, if the sampling
amount $N \geq 3\epsilon^{-2} \sr(\ma) \ln(4\sr(\ma)/\delta)$, then with a probability at least $1-\delta$ 
\[ \|\ma\ma^T - (\ma\ms)(\ma\ms)^T\|_2 \le \|\ma\ma^T\|_2.\] 
Thus, if the stable rank $\sr(\ma)$ is much smaller than the large 
dimension~$n$, then one can compute an accurate approximation to the Gram matrix $\ma\ma^T$ from just a few columns of~$\ma$.  
}


\Ilse{\paragraph{Covariance Estimation} 
 Consider $N$
independent and identically distributed samples $\vx_1,\dots,\vx_N$, drawn from a Gaussian distribution with zero mean and covariance $\msig$. To approximate~$\msig$, one can compute an unbiased
estimator using the Monte Carlo approximation
$\msig_N \equiv \frac1N \sum_{j=1}^N \vx_j \vx_j^T$.
An important question in probability and statistics is to determine the sampling amount $N$ so that
\begin{equation}\label{eqn:coverror} 
\E\left[ \| \msig - \msig_N\|_2 \right]\leq \varepsilon \|\msig\|_2.
\end{equation}
That is, the expectation of the relative error in the approximation~$\msig_N$ in the two-norm should not exceed~$\varepsilon$. By~\cite[Equation (1.1)]{koltchinskii2017concentration} implies 
\begin{align*}
\E\left[\| \msig - \msig_N\|_2\right] \leq C \left( \sqrt{\frac{\intdim(\msig)}{N}} + \frac{\intdim(\msig)}{N} \right) \|\msig\|_2,
\end{align*}
where $C$ is an absolute constant.
Therefore, a sampling amount of $N \sim \varepsilon^{-2} \intdim(\msig)$ is sufficient to ensure~\eqref{eqn:coverror}. If the intrinsic dimension $\intdim(\msig) \ll n$, then $\msig$ can be approximated with only a few samples.  } 

\Ilse{\paragraph{Deep Neural Networks (DNN)}
Consider a fully connected feed-forward DNN with Lipschitz activation $\phi$, and $L$ layers, where
each layer is represented by a 
weight matrix $\mw^{(\ell)}$, a 
bias $\vb^{(\ell)}$, 
and a normalization factor $\gamma^{(\ell)}$,
$1\leq \ell\leq L$ \cite[Section 2]{GFMA21}.
The feed-forward mappings applied to an input vector $\vx$ 
are defined as
\begin{align*}
\valpha^{(0)}(\vx)&\equiv \vx, \qquad
\valpha^{(\ell)}(\vx)\equiv 
\phi\left(\gamma^{(\ell)} \mw^{(\ell)}
\valpha^{(\ell-1)}(\vx)+\vb^{(\ell)}\right),
\qquad 1\leq \ell\leq L.
\end{align*}
We consider the generalization error and noise stability
of DNNs.}

\Ilse{The `generalization error' reflects the accuracy
with which the DNN makes predictions for new, unseen data \cite{AGNZ18}. One can bound the generalization error in terms of the expression 
\cite[Theorem 2.2]{AGNZ18} 
\begin{align}\label{e_ge}
\prod_{\ell=1}^{L}{\|\mw^{(\ell)}\|_2}
\sum_{\ell=1}^L{\sr(\mw^{(\ell)})}.
\end{align}
As a consequence, decreasing the stable ranks of the
weight matrices $\mw^{(\ell)}$ can decrease the
generalization error of the  DNN and improve its classification accuracy~\cite{STD20}.
That is why the sum of the stable ranks 
represents `a natural measure of the true parameter count'
of the DNN \cite[Section 2.2]{AGNZ18}.
}

\Ilse{The `noise stability' of a DNN reflects the effect of noise injected in a layer \cite[Section 3]{AGNZ18}.
Specifically, the noise sensitivity of a linear map $\mw\in\rmn$ 
with respect to standard Gaussian noise $\mathcal{N}(\vzero, \mi_n)$ at a nonzero vector
$\vx\in\rn$ is defined as \cite[Definition~3]{AGNZ18}
\begin{align*}
\Psi(\mw,\vx) \equiv\E_{\veta\sim\mathcal{N}(\vzero,\mi)}
\left[\frac{\|\mw(\vx+\veta)-\mw\vx\|_2^2}{\|\mw\vx\|_2^2}\right],
\end{align*}
that is, the expectation of the relative change in 
$\mw\vx$ in the two-norm.}

\Ilse{The noise sensitivity of $\mw$ is bounded below by the stable rank of $\mw$ \cite[Proposition 3.1]{AGNZ18},
\begin{align*}
\Psi(\mw,\vx)=\frac{\|\mw\|_F^2\|\vx\|_2^2}{\|\mw\vx\|_2^2}\geq \sr(\mw).
\end{align*}
}

\Ilse{\paragraph{Collaborative Filtering (CF)}
This is a type of recommender
system that selects items for a particular user 
based on past recommendations for other users. 
The cost of working with the large dimensional user-item matrix
can be reduced by approximately factoring it into two separate,
lower dimensional user and item embedding matrices. 
It can be shown \cite{LWZ2024} that the stable rank of these embedding matrices
is highly correlated with CF performance. In particular and in contrast to DNNs, higher stable rank is better,
because it correlates with higher performing loss functions \cite[Section 3.2]{LWZ2024}.}

\subsection{Contributions}
We derive numerous properties of the stable rank and intrinsic dimension. Our proofs rely on basic results from matrix analysis such as eigenvalue and singular value inequalities.

\begin{enumerate}
\item Unlike the classical rank, the stable rank
and intrinsic dimension can increase when deleting a row and/or a column from the matrix (Section~\ref{s_delete}).
\item Unlike the classical rank, the stable rank
does not satisfy the rank-sum inequality, but the
intrinsic dimension does (Section~\ref{s_ranksumin}).
\item Like the classical rank, the stable rank and intrinsic dimension of a Hermitian positive semi-definite
matrix can increase by at most one, if a Hermitian positive
semi-definite matrix of rank~1 is added.
\item Unlike the classical rank, the stable rank and
intrinsic dimension do not satisfy the rank-product
inequality (Section~\ref{s_rankprod}).
\item Unlike the classical rank, multiplication
by a nonsingular matrix can increase the stable rank 
to its maximum, or decrease it to its minimum (Section~\ref{s_nonsing}).
\item Unlike the classical rank, the rank of the cross
product matrix\Ilse{\footnote{The cross product matrix of a matrix $\ma$ is $\ma^*\ma$ \cite[Section 2]{Stewart1984}, \cite[Section 1]{Stewart1987}.}}
can be strictly smaller than the rank of the
matrix itself (Section~\ref{s_cross}).

\item We generalize the stable rank to the $p$-stable rank in a general Schatten $p$-norm (Definition~\ref{d_pstable}), thereby unifying the concept of stable rank and intrinsic dimension: The stable rank represents the 2-stable rank, while the intrinsic dimension is the 1-stable rank of a Hermitian positive semi-definite matrix (Remark~\ref{r_unify}).
\item We derive sum inequalities  
for the  $p$th root of the $p$-stable rank 
(Theorem~\ref{t_sr8}, Corollary~\ref{c_sr8}, 
and Theorem~\ref{t_sr9}), and product inequalities
for the $p$-stable rank
(Theorem~\ref{t_srp1}, Corollary~\ref{c_srp1} and Theorem~\ref{t_eq7}).
\item We show that the $p$th root of the $p$-stable rank is well-conditioned in the norm-wise absolute sense
(Theorem~\ref{t_cond1}). In particular, the intrinsic dimension and the square root of the 2-stable rank are well-conditioned (Corollaries \ref{c_cond1} and~\ref{c_cond2}).
\end{enumerate}

\subsection{Overview}
After reviewing the basic properties of the stable rank and intrinsic dimension (Section~\ref{s_prop}), we present
simple examples to illustrate that the stable rank satisfies neither the rank inequality for
deleting a row or column, nor the rank-sum
or the rank-product inequalities, while the intrinsic dimension satisfies only two of the
rank-sum inequalities (Section~\ref{s_rankineq}).
Furthermore, we illustrate that the stable rank and intrinsic dimension satisfy neither the rank product equality
for non-singular matrices
nor the matrix cross product equality (Section~\ref{s_rankeq}).
At last, we extend the notion of stable rank and intrinsic dimension to the $p$-stable rank, derive
sum inequalities for the $p$th root of the $p$-stable
rank and
product inequalities for the $p$-stable rank, and
show that the $p$th root of the $p$-stable rank
is well-conditioned in the
norm-wise absolute sense (Section~\ref{s_ext}).

\subsection{Notation and auxiliary results}
The $n\times n$ identity matrix is $\mi_n$.
The conjugate transpose 
of $\ma\in\cmn$ is 
$\ma^*\in\cnm$.
The singular values of $\ma\in\cmn$ are 
$\sigma_1(\ma)\geq \cdots \geq \sigma_{\min\{m,n\}}\geq 0$.
The two-norm condition number of a non-singular matrix $\ma\in\cnn$ is
$\kappa_2(\ma)=\|\ma\|_2\|\ma^{-1}\|_2=\sigma_1(\ma)/\sigma_n(\ma)$.
The eigenvalues of a Hermitian matrix $\ma\in\cnn$ are 
$\lambda_1(\ma)\geq \cdots\geq \lambda_n(\ma)$.

We exploit the unitary invariance of the 
two- and
Frobenius norms. 
That is, if $\ma\in\cmn$ and $\mq\in\complex^{k\times m}$
with $\mq^*\mq=\mi_m$, then $\|\mq\ma\|_2=\|\ma\|_2$
and $\|\mq\ma\|_F=\|\ma\|_F$.

For Hermitian positive semi-definite matrices
$\ma,\mb\in\cnn$,
Weyl's monotonicity theorem 
\cite[Section 10.3]{Par98} implies
\begin{align}\label{e_weyl}
\|\ma+\mb\|_2=\lambda_1(\ma+\mb)\geq \lambda_1(\ma)+\lambda_n(\mb)\geq 
\lambda_1(\ma)=\|\ma\|_2.
\end{align}

\section{Properties of the stable rank and intrinsic dimension}\label{s_prop}
We review basic properties of the stable rank (Section~\ref{s_st})
and intrinsic dimension (Section~\ref{s_id}),
and relate the two (Section~\ref{s_idsr}).

\subsection{Stable rank}\label{s_st}
The stable rank reflects how fast the singular values of a matrix decrease from largest to smallest.

\begin{definition}[Remark 1.3 in \cite{rudelson2007sampling}, Section 2.1.5 in \cite{Tropp2015}, Definition 7.6.7 in \cite{Versh18}]
The {\rm stable rank} of a non-zero matrix $\ma\in\cmn$ is
\begin{align*}
\sr(\ma)\equiv\|\ma\|_F^2/\|\ma\|_2^2.
\end{align*}
The stable rank of the zero matrix is
$\sr(\vzero_{m\times n})=0$.
\end{definition}

We can relate the stable rank to the rank via
$\|\ma\|_F^2\leq \rank(\ma)\|\ma\|_2^2$ \cite[P2.4.7]{GovL13}.
\Ilse{If $\ma\neq \vzero$, then}
\begin{align*}
1\leq\sr(\ma)\leq \rank(\ma)\leq \min\{m,n\}.
\end{align*}
\Ilse{An immediate consequence is a potentially
much tighter relation between Frobenius and two norms,
\begin{align}\label{e_ftnorm}
\|\ma\|_F\leq\sqrt{\sr(\ma)} \>\|\ma\|_2.
\end{align}
This, in turn, implies the improvement
\begin{align}\label{e_ftnorm2}
\|\ma\mb\|_F \le \sqrt{\min\{\sr(\ma),\sr(\mb),\sr(\ma\mb)\}}\|\ma\|_2\|\mb\|_2
\end{align}
over the traditional submultiplicative bound
\begin{align*}
\|\ma\mb\|_F\leq \sqrt{\min\{\rank(\ma),\rank(\mb),\rank(\ma\mb) \}} \|\ma\|_2\|\mb\|_2 .
\end{align*}
}

\begin{example}\label{ex_sr2}
The following are special cases where $\sr(\ma)=\rank(\ma)$.
\begin{itemize}
\item If $\rank(\ma)=1$ then $\sr(\ma)=1$. 
\item If $\ma=\alpha\mU$ for some $\alpha\neq 0$ and unitary, or real orthogonal, $\mU\in\cnn$, then
$\sr(\ma)=n=\rank(\ma)$.
\item If $\ma\in\cmn$ with $\rank(\ma)=r\geq 1$
has singular values $\sigma_1(\ma)=\cdots=\sigma_r(\ma)>0$,
then $\sr(\ma)=r=\rank(\ma)$.
\item If $\ma$ is an orthogonal projector, 
then $\sr(\ma) = \|\ma\|_F^2 = \rank(\ma)$.
\end{itemize}
\end{example}

The stable rank of a block diagonal matrix is bounded above by the sums of the stable ranks of its diagonal blocks. 

\begin{example}[Section 7.3.3 in \cite{Tropp2015}]
Let
\begin{align*}
\ma=\begin{bmatrix}\ma_{11} & \\ & \ma_{22}\end{bmatrix}
\end{align*}
where $\ma_{11}\in\complex^{k\times k}$ and $\ma_{22}\in\complex^{\ell\times \ell}$. Then
\begin{align*}
\min\{\sr(\ma_{11}),\sr(\ma_{22})\}\leq \sr(\ma)\leq \sr(\ma_{11})+\sr(\ma_{22}).
\end{align*}
This follows from the elements of $\ma_{11}$ and $\ma_{22}$ occupying
distinct positions so that
\begin{align*}
\sr(\ma)=\frac{\|\ma\|_F^2}{\|\ma\|_2^2}=
\frac{\|\ma_{11}\|_F^2+\|\ma_{22}\|_F^2}{\max{\{\|\ma_{11}\|_2^2, \|\ma_{22}\|_2^2\}}}\leq \sr(\ma_{11})+\sr(\ma_{22}).
\end{align*}
\end{example}

\subsection{Intrinsic dimension}\label{s_id}
The intrinsic dimension of a Hermitian positive semi-definite matrix can be interpreted as a stable rank.

\begin{definition}[\cite{HKZ2012}, Definition 7.1.1 in \cite{Tropp2015},
Remark~5.6.3 in \cite{Versh18}]
The {\rm intrinsic dimension} of a non-zero Hermitian positive semi-definite matrix $\ma\in\cnn$ is
\begin{align*}
\intdim(\ma)\equiv\trace(\ma)/\|\ma\|_2.
\end{align*}
The intrinsic dimension of the zero matrix is
$\intdim(\vzero_{n\times n})=0$.
\end{definition}

The intrinsic dimension of \Ilse{a non-zero matrix} $\ma\in\cnn$ is related to the rank via
\begin{align*}
1\leq \intdim(\ma)\leq \rank(\ma)\leq n.
\end{align*}

\begin{example}
The following are special cases where $\intdim(\ma)=\rank(\ma)$.
\begin{itemize}
\item If $\rank(\ma)=1$ then $\intdim(\ma)=1$. 
\item If $\ma=\alpha\mi_n$ for some $\alpha> 0$, then $\intdim(\ma)=n=\rank(\ma)$.
\item If $\ma\in\cnn$ is Hermitian positive semi-definite with $\rank(\ma)=r\geq 1$ and
has eigenvalues $\lambda_1(\ma)=\cdots=\lambda_r(\ma)>0$,
then $\intdim(\ma)=r=\rank(\ma)$.
\item If $\ma$ is an orthogonal projector, then $\intdim(\ma) = \trace(\ma) = \rank(\ma)$.
\end{itemize}
\end{example}

The intrinsic dimension of a Hermitian
positive semi-definite matrix is bounded above by the sums of the intrinsic
dimensions of its diagonal blocks. 

\begin{example}\label{ex_intdim0}
 Let
\[ \ma = \bmat{\ma_{11} & \ma_{12} \\ \ma_{12}^* & \ma_{22} }\in\cnn\]
be Hermitian positive semi-definite, with
$\ma_{11}\in\complex^{k\times k}$ and $\ma_{22}\in\complex^{(n-k)\times (n-k)}$.
Then
\[  \intdim(\ma) \le \intdim(\ma_{11}) + \intdim(\ma_{22}).   \] 
This follows from the Hermitian positive semi-definiteness
of the principal submatrices
$\ma_{11}$ and $\ma_{22}$, the linearity of the trace, 
\begin{align*}
\trace(\ma) = \trace(\ma_{11}) + \trace(\ma_{22}),
\end{align*}
and
$\max\{\|\ma_{11}\|_2,\|\ma_{22}\|_2\} \leq \|\ma\|_2$.

The special case $\ma_{12}=\vzero$ is shown in 
\cite[Section 7.3.1]{Tropp2015}.
\end{example}

\subsection{Relation between intrinsic dimension and stable rank}\label{s_idsr}
The properties of the two-norm and Frobenius norm imply \cite[Section 7.2.2]{Tropp2015},
\begin{align*}
\intdim(\ma^*\ma)=\intdim(\ma\ma^*)=\sr(\ma).
\end{align*}
If $\ma\in\cnn$ is Hermitian positive semi-definite, then \cite[Section 6.5.4]{Tropp2015}
\begin{align*}
\intdim(\ma)=\sr(\ma^{1/2}),
\end{align*}
where $\ma^{1/2}$ is the Hermitian positive semi-definite
square root of $\ma$. 

More general relations are presented in Remark~\ref{r_idsr}.

\section{Rank inequalities}\label{s_rankineq}
We illustrate that the stable rank satisfies neither the rank inequality for
deleting a row or column (Section~\ref{s_delete}),
nor the rank-sum  (Section~\ref{s_ranksum})
nor the rank-product inequalities (Section~\ref{s_rankprod}), while the intrinsic dimension satisfies only two of the rank-sum inequalities.

\subsection{Deleting rows or columns from a matrix}\label{s_delete}
If a column or row is deleted from $\ma\in\cmn$, then the
rank of the resulting submatrix cannot exceed the rank
of the original matrix \cite[Section 0.4.5(b)]{HoJoI}.
In other words, if $\ma=\begin{bmatrix}\widehat{\ma} &\va \end{bmatrix}\in\cmn$ where $\va\in\cm$, then
\begin{align}\label{e_del}
\rank(\widehat{\ma})\leq \rank(\ma).
\end{align}

The stable rank and intrinsic dimension do, in general, not satisfy (\ref{e_del}).

\begin{example}\label{ex_sr3}
Deleting a column or row from a matrix can increase the stable
rank and the intrinsic dimension.

Let $n \ge 3$ and define
\begin{align*}
\ma=\begin{bmatrix} \mi_{n-1} & \\ & \alpha\end{bmatrix}
\in\rnn\qquad\text{with}\qquad \alpha>1.
\end{align*}

Deleting the trailing column from $\ma$ gives
\begin{align*}
\widehat{\ma} =\begin{bmatrix} \mi_{n-1}\\ \vzero_{1\times (n-1)}\end{bmatrix}\in\real^{n\times (n-1)}.
\end{align*}
The stable ranks are
\begin{align*}
\sr(\ma)=\frac{n-1+\alpha^2}{\alpha^2}=1+\frac{n-1}{\alpha^2},\qquad \sr(\widehat{\ma})=n-1.
\end{align*}
If $\alpha>\sqrt{\frac{n-1}{n-2}}$ then
\begin{align*}
\sr(\widehat{\ma})> \sr(\ma).
\end{align*}

Deleting the trailing row and column from $\ma$ gives $\widehat{\ma} =\mi_{n-1}$. The intrinsic dimensions are
\begin{align*}
\intdim(\ma)=\frac{n-1+\alpha}{\alpha}=1+\frac{n-1}{\alpha},\qquad \intdim(\widehat{\ma})=n-1.
\end{align*}
If $\alpha>\frac{n-1}{n-2}$ then
\begin{align*}
\intdim(\widehat{\ma})> \intdim(\ma).
\end{align*}
\end{example}

 \subsection{Rank of a sum}\label{s_ranksum}
If $\ma,\mb\in\cmn$ then the following hold
\cite[Section 0.4.5(d)]{HoJoI} 
\begin{enumerate}
\item {\rm Rank-sum inequality}
\begin{align}\label{e_sum1}
\rank(\ma+\mb)\leq \rank(\ma)+\rank(\mb).
\end{align}
\item {\rm Adding a matrix of rank~1.
} If $\rank(\mb)=1$, then
\begin{align}\label{e_sum3}
-1\leq \rank(\ma+\mb)-\rank(\ma)\leq 1.
\end{align}
\end{enumerate}

\subsubsection{Rank-sum inequality}\label{s_ranksumin}
\Ilse{In the context of DNNs with weight matrices of the same dimension, a lower bound on the 
generalization error (\ref{e_ge}) might be possible if the stable rank
satisfied the rank-sum inequality~(\ref{e_sum1}),
allowing one to bound $\sum_{\ell=1}^L{\sr(\mw^{(\ell)})}$ below by
$\sr\left(\sum_{\ell=1}^L{\mw^{(\ell)}}\right)$.
However, this is not possible.}

We illustrate that the stable rank, in general, does not 
satisfy the rank-sum inequality~(\ref{e_sum1}), but
the intrinsic dimension does.

\begin{example}
The stable rank does not satisfy the rank-sum inequality~(\ref{e_sum1}).

Let $\ma,\mb\in\rnn$ with $n\ge 4$ and
\begin{align*}
\ma=\begin{bmatrix} \alpha & \\ & 2\mi_{n-1}\end{bmatrix}, \qquad
\mb=\begin{bmatrix} -\alpha & \\ & -\mi_{n-1}\end{bmatrix},\qquad
\ma+\mb=\begin{bmatrix} 0 & \\ & \mi_{n-1}\end{bmatrix}
\end{align*}
where $|\alpha|>2$. The stable ranks are
\begin{align*}
\sr(\ma)=1+4\frac{n-1}{\alpha^2},\qquad \sr(\mb)=1 +\frac{n-1}{\alpha^2},\qquad
\sr(\ma+\mb)=n-1.
\end{align*}
If $\alpha^2>5 \frac{n-1}{n-3}$ then
\begin{align*}
\sr(\ma+\mb)\geq \sr(\ma)+\sr(\mb).
\end{align*}
\end{example}

The intrinsic dimension does satisfy the rank-sum inequality~(\ref{e_sum1}).

\begin{theorem}\label{thm:intdimsubadditive}
If $\ma,\mb\in\cnn$ are Hermitian positive semi-definite, then 
\begin{align*}
\intdim(\ma+\mb)\leq \intdim(\ma)+\intdim(\mb).
\end{align*}
\end{theorem}

\begin{proof}
\Ilse{If either or both $\ma$ and $\mb$ are zero, then the bound holds with equality. Therefore, without loss of generality, assume both $\ma$ and $\mb$ are not zero. }
For the numerator of the intrinsic dimension, the linearity of the trace implies $\trace(\ma+\mb)=\trace(\ma)+\trace(\mb)$,
while Weyl's monotonicity theorem~(\ref{e_weyl})
 implies for the denominator
$\|\ma+\mb\|_2\geq\|\ma\|_2$ and
$\|\ma+\mb\|_2\geq \|\mb\|_2$.
\end{proof}

\subsubsection{Adding a matrix of rank~1}
The intrinsic dimension satisfies  the second
inequality in (\ref{e_sum3}), but not the first.

The
second inequality in (\ref{e_sum3}) implies that adding a 
Hermitian positive semi-definite matrix
of rank~1 can increase the intrinsic dimension by at most one.

\begin{theorem}\label{t_intdim9}
Let $\ma,\mb\in\cnn$ be Hermitian positive semi-definite.
If\\ $\rank(\mb)=1=\intdim(\mb)$ then 
\Ilse{\begin{align*}
1\leq \intdim(\ma+\mb)\leq \intdim(\ma)+1.
\end{align*}}
\end{theorem}

\begin{proof}
\Ilse{Since $\ma$ and $\mb$ are Hermitian positive semi-definite, so is
$\ma+\mb$. In addition, $\rank(\mb)=1$ implies
$\rank(\ma+\mb)\geq \rank(\mb)=1$. Thus $\intdim(\ma+\mb)\geq 1$.
This establishes the lower bound.
}

\Ilse{If $\ma=\vzero$ then $\intdim(\ma)=0$ and 
$\intdim(\ma+\mb)=\intdim(\mb)=1$, so that the upper bound 
holds with equality.}

\Ilse{Now assume that $\ma\neq \vzero$.}
Applying Weyl's monotonicity theorem (\ref{e_weyl}),
$\|\ma+\mb\|_2\geq \|\ma\|_2$ and $\|\ma+\mb\|_2\geq \|\mb\|_2$ gives
\begin{align*}
\intdim(\ma+\mb)-\intdim(\ma)&=
\frac{\trace(\ma)}{\|\ma+\mb\|_2}+
\frac{\trace(\mb)}{\|\ma+\mb\|_2}-\frac{\trace(\ma)}{\|\ma\|_2}\\
&\leq\frac{\trace(\ma)}{\|\ma\|_2}+
\frac{\trace(\mb)}{\|\mb\|_2}-\frac{\trace(\ma)}{\|\ma\|_2}\\
&=\intdim(\mb)=1.
\end{align*}
\end{proof}

\Ilse{If $\ma=\vzero$, then the lower and upper bounds in 
Theorem~\ref{t_intdim9} hold with equality.}
The following example illustrates that
adding a Hermitian positive semi-definite matrix of 
rank~1 can 
lower the intrinsic dimension by more than one, \Ilse{and that 
for $\ma\neq \vzero$
the intrinsic dimension of the sum in Theorem~\ref{t_intdim9}
can be arbitrarily close to one.}

\begin{example}
The intrinsic dimension does not satisfy the first inequality in~(\ref{e_sum3}).

Let $\ma,\mb\in\rnn$ with $n \ge 4$ and
\begin{align*}
\ma=\begin{bmatrix} 0& \\ &\mi_{n-1}\end{bmatrix}, \qquad
\mb=\begin{bmatrix} \beta & \\ & \vzero_{(n-1)\times (n-1)}\end{bmatrix},
\qquad
\ma+\mb=\begin{bmatrix} \beta & \\ &\mi_{n-1}\end{bmatrix}.
\end{align*}
where $\beta>1$ and $\intdim(\mb)=1=\rank(\mb)$.

The intrinsic dimensions are
\begin{align*}
\intdim(\ma)=n-1, \qquad 
\intdim(\ma+\mb)=\frac{\beta + (n-1)}{\beta}=1+\frac{n-1}{\beta}.
\end{align*}
If $\beta>\frac{n-1}{n-3}$ then 
\Ilse{\begin{align*}
\intdim(\ma+\mb)< \intdim(\ma)-1.
\end{align*}
Furthermore, $\intdim(\ma+\mb)\rightarrow 1$ as 
$\beta\rightarrow\infty$.
}
\end{example}

\subsection{Rank of a product}\label{s_rankprod}
If $\ma\in\cmn$ and $\mb\in\complex^{n\times k}$, then \cite[Section 0.4.5(c)]{HoJoI} 
\begin{align}\label{e_prod1}
\rank(\ma\mb)\leq \min\{\rank(\ma), \rank(\mb)\}
\end{align}
We illustrate that the stable rank and intrinsic dimension, in general, do not
satisfy~(\ref{e_prod1}).

\begin{example}\label{ex_sr1}
The stable rank and intrinsic dimension do not satisfy (\ref{e_prod1}).

Let $\ma\in\rnn$ with
\begin{align*}
\ma=\begin{bmatrix} \mi_{n-1} & \\ & \alpha\end{bmatrix}, \qquad 
\mb=\begin{bmatrix} \mi_{n-1} & \\ &1/\alpha\end{bmatrix},\qquad
\ma\mb=\mi_n,
\end{align*}
where $\alpha>1$.
The stable ranks are
\begin{align*}
\sr(\ma)=\frac{n-1+\alpha^2}{\alpha^2}=1+\frac{n-1}{\alpha^2}, \qquad
\sr(\mb)=n-1+\frac{1}{\alpha^2},\qquad \sr(\ma\mb)=n.
\end{align*}
From $\alpha>1$ follows 
\begin{align*}
\sr(\ma\mb)>\sr(\ma),\qquad 
\sr(\ma\mb)> \sr(\mb).
\end{align*}
In summary,
\begin{align*}
\sr(\ma\mb)& >\max\{\sr(\ma),\sr(\mb)\}.
\end{align*}
The intrinsic dimensions are
\begin{align*}
\intdim(\ma)=\frac{n-1+\alpha}{\alpha}=1+\frac{n-1}{\alpha}, \quad
\intdim(\mb)=n-1+\frac{1}{\alpha},\quad \intdim(\ma\mb)=n.
\end{align*}
From $\alpha>1$ follows 
\begin{align*}
\intdim(\ma\mb)>\intdim(\ma),\qquad 
\intdim(\ma\mb)>\intdim(\mb).
\end{align*}
In summary,
\begin{align*}
\intdim(\ma\mb)&> \max\{\intdim(\ma),\intdim(\mb)\}.
\end{align*}
\end{example}

The stable rank satisfies (\ref{e_prod1}) in the special case where $\ma$ or $\mb$
is unitary, or real orthogonal. 
The intrinsic dimension satisfies (\ref{e_prod1}) in the special case where $\ma$ or $\mb$ equals $\alpha\mi_n$ with $\alpha>0$. 
\Ilse{In general, for $\intdim(\ma\mb)$ to be defined, 
the product $\ma\mb$
must also be Hermitian positive definite, which is the case when
 $\ma$ and $\mb$ are simultaneously diagonalizable.}

\section{Rank equalities}\label{s_rankeq}
We illustrate that the stable rank and intrinsic dimension satisfy neither the rank product equality
for non-singular matrices (Section~\ref{s_nonsing})
nor the cross product equality (Section~\ref{s_cross}).

\subsection{Multiplication by nonsingular matrices}\label{s_nonsing}
Multiplication by nonsingular matrices does not change the rank. That is, let $\ma\in\cmn$. If $\mb\in\cmm$ and $\mc\in\cnn$ are nonsingular, then \cite[Section 0.4.6(b)]{HoJoI}
\begin{align}\label{e_eq1}
\rank(\ma)=\rank(\mb\ma\mc).
\end{align}

The stable rank and intrinsic dimension do, in general, not satisfy (\ref{e_eq1}).
More specifically,
the next two examples illustrate how multiplication with a nonsingular matrix can make the stable rank as large as
or as small as possible.

\begin{example}\label{ex_eqsr1}
For any $\ma\in\cnn$ with $\rank(\ma)=r\geq 1$, 
there is a nonsingular matrix $\mb$ that increases the stable rank to the maximal amount $\sr(\ma\mb)=r$.

Let $\ma=\mU\diag\begin{pmatrix}\msig & \vzero_{(n-r)\times (n-r)}\end{pmatrix}\mv^*$ be a SVD where $\mU,\mv\in\cnn$ are unitary, and $\msig\in\real^{r\times r}$ 
is nonsingular
diagonal. Set
\begin{align*}
\mb=\mv\diag\begin{pmatrix}\msig^{-1} & \mi_{n-r}\end{pmatrix},\qquad
\ma\mb=\mU\diag\begin{pmatrix}\mi_r& \vzero_{(n-r)\times (n-r)}\end{pmatrix}.
\end{align*}
Then
$\sr(\ma)\leq \sr(\ma\mb)=r$.
\end{example}

The minimal stable rank of a non-zero matrix $\ma$ is
$\sr(\ma)=1=\rank(\ma)$. For matrices 
with $rank(\ma)\geq 2$, a matrix multiplication
can bring the stable
rank as close to~1 as possible.

\begin{example}\label{ex_eqsr2}
For any $\ma\in\cnn$ with $\rank(\ma)\geq 2$, 
there is a nonsingular matrix $\mb$ that decreases the stable rank arbitrarily close to~1.

Let 
\begin{align*}
\ma=\mU\diag\begin{pmatrix}\sigma_1&\cdots &\sigma_r & 0&
\cdots &0\end{pmatrix}\mv^* \qquad \text{with}\quad
\sigma_1\geq \cdots \geq\sigma_r>0
\end{align*}
be a SVD where $\mU,\mv\in\cnn$ are unitary.
Pick some $0<\alpha<1$ and set
\begin{align*}
\mb&=\mv\diag\begin{pmatrix}\frac{1}{\sigma_1}& \frac{\alpha}{\sigma_2} & \cdots&
\frac{\alpha}{\sigma_r} & 1 &\cdots &1\end{pmatrix}\\
\ma\mb&=\mU\diag\begin{pmatrix}1 & \alpha& \cdots &\alpha& 
0&\cdots &0\end{pmatrix}.
\end{align*}
Then
\begin{align*}
\sr(\ma\mb)=1+(r-1)\alpha^2\rightarrow 1\qquad \text{as}\qquad \alpha\rightarrow 0.
\end{align*}
\end{example}

\begin{remark}
The stable rank is unitarily invariant. That is, if 
$\ma\in\cmn$, and $\mU\in\cmm$ and $\mv\in\cnn$ are unitary or real orthogonal, then
\begin{align*}
\sr(\ma)=\sr(\mU\ma\mv).
\end{align*}
\end{remark}

Analogous to Examples~\ref{ex_eqsr1} and~\ref{ex_eqsr2},
we 
illustrate below how a congruence transformation can make the intrinsic dimension as large or as
small as possible.

\begin{example}\label{ex_eqid1}
For any Hermitian positive semi-definite matrix $\ma\in\cnn$ with $\rank(\ma)=r\geq 1$, 
there is a  congruence transformation with a 
nonsingular matrix~$\mb$
that increases the intrinsic dimension to the maximal amount $\intdim(\mb^*\ma\mb)=r$.

Let $\ma=\mv\diag\begin{pmatrix}\msig & \vzero_{(n-r)\times (n-r)}\end{pmatrix}\mv^*$ be an eigenvalue
decomposition where $\mv\in\cnn$ is unitary, and $\msig\in\real^{r\times r}$ 
is nonsingular
diagonal. Set
\begin{align*}
\mb=\mv\diag\begin{pmatrix}\msig^{-1/2} & \mi_{n-r}\end{pmatrix}\mv^*,\qquad
\mb^*\ma\mb=\mv\diag\begin{pmatrix}\mi_r& \vzero_{(n-r)\times (n-r)}\end{pmatrix}\mv^*.
\end{align*}
Then
$\intdim(\ma)\leq \intdim(\mb^*\ma\mb)=r$.
\end{example}

Like the stable rank,
the minimal intrinsic dimension of a non-zero 
Hermitian positive semi-definite matrix $\ma$ is
$\intdim(\ma)=1=\rank(\ma)$. For matrices 
with $rank(\ma)\geq 2$, a congruence transformation
can bring the intrinsic dimension
as close to~1 as possible.

\begin{example}\label{ex_eqid2}
For any Hermitian positive semi-definite matrix
$\ma\in\cnn$ with $\rank(\ma)\geq 2$, 
there is a congruence transformation that decreases the 
intrinsic dimension arbitrarily close to~1.

Let 
\begin{align*}
\ma=\mv\diag\begin{pmatrix}\sigma_1&\cdots &\sigma_r & 0&
\cdots &0\end{pmatrix}\mv^* \qquad \text{with}\quad
\sigma_1\geq \cdots \geq\sigma_r>0
\end{align*}
be an eigenvalue decomposition where $\mv\in\cnn$ is unitary.
Pick some $0<\alpha<1$ and set
\begin{align*}
\mb&=\mv\diag\begin{pmatrix}\sqrt{\frac{1}{\sigma_1}}& \sqrt{\frac{\alpha}{\sigma_2}} & \cdots&
\sqrt{\frac{\alpha}{\sigma_r}} & 1 &\cdots &1\end{pmatrix}\\
\mb^*\ma\mb&=\mv\diag\begin{pmatrix}1 & \alpha& \cdots &\alpha& 
0&\cdots &0\end{pmatrix}\mv^*.
\end{align*}
Then
\begin{align*}
\intdim(\mb^*\ma\mb)=1+(r-1)\alpha\rightarrow 1\qquad \text{as}\qquad \alpha\rightarrow 0.
\end{align*}
\end{example}

\begin{remark}\label{r_eqid1}
The intrinsic dimension is invariant under unitary (or real orthogonal) similarity transformations. That is, if $\ma\in\cnn$ is Hermitian positive semi-definite, and $\mU\in\cnn$ is unitary or real orthogonal matrix, then
\begin{align*}
\intdim(\ma)=\intdim(\mU^*\ma\mU).
\end{align*}
\end{remark}

\subsection{Rank of a cross product matrix}\label{s_cross}
The rank of a cross product matrix is equal to the rank of the original matrix. That is, if $\ma\in\cmn$ then
\cite[Section~0.4.6.(d)]{HoJoI}
\begin{align}\label{e_eq2}
\rank(\ma^*\ma)=\rank(\ma).
\end{align}

Neither the stable rank nor the intrinsic dimension satisfy~(\ref{e_eq2}). They are bounded by but
not equal to the stable rank or intrinsic dimension of the cross product matrix.

\begin{theorem}\label{t_eq1}
If $\ma\in\cmn$ is non-zero, then
\begin{align*}
\sr(\ma^*\ma)\leq \sr(\ma)\qquad
\text{and}\qquad \sr(\ma\ma^*)\leq \sr(\ma).
\end{align*}
If $\ma\in\cnn$ is non-zero Hermitian positive semi-definite, then
\begin{align*}
\intdim(\ma^*\ma)=\sr(\ma)\leq \intdim(\ma).
\end{align*}
\end{theorem}

\begin{proof}
The inequalities for the stable rank follow from the 
strong submultiplicativity $\|\ma^*\ma\|_F\leq
\|\ma\|_2\|\ma\|_F$, and 
$\|\ma^*\ma\|_2=\|\ma\|_2^2=\|\ma\ma^*\|_2$.

The equality for the intrinsic dimension follows
from Section~\ref{s_idsr}, while the inequality 
follows from $\trace(\ma^*\ma)=\trace(\ma^2)\leq \|\ma\|_2\trace(\ma)$.
\end{proof}

\begin{example}
The stable rank and intrinsic dimension do not satisfy (\ref{e_eq2}), that is, the stable rank and intrinsic
dimension of a cross product matrix
can be strictly smaller than those of the matrix.

Let $\ma\in\rnn$ with
\begin{align*}
\ma=\begin{bmatrix} 1 & \\ & \alpha \mi_{n-1}\end{bmatrix}
\qquad \text{and}\qquad
\ma^*\ma =\ma^2=\begin{bmatrix}1 & \\ & \alpha^2\mi_{n-1}
\end{bmatrix}
\end{align*}
where $0<\alpha<1$.

The stable ranks are
\begin{align*}
\sr(\ma)=1+(n-1)\alpha^2\qquad \text{and}\qquad 
\sr(\ma^*\ma)=1+(n-1)\alpha^4.
\end{align*}
From $\alpha<1$ follows
\begin{align*}
sr(\ma^*\ma)<\sr(\ma).
\end{align*}
Similarly, since $\ma$ is Hermitian positive definite,
the intrinsic dimensions are
\begin{align*}
\intdim(\ma)=1+(n-1)\alpha\qquad \text{and}\qquad 
\intdim(\ma^*\ma)=1+(n-1)\alpha^2.
\end{align*}
From $\alpha<1$ follows
\begin{align*}
\intdim(\ma^*\ma)<\intdim(\ma).
\end{align*}
\end{example}

\begin{remark}
The intrinsic dimension
of a Cholesky factor or square root can exceed that
of the original matrix. 
That is, if $\ma\in\cnn$
is Hermitian positive semi-definite, with pivoted
Cholesky factorization $\mP^*\ma\mP=\ml\ml^*$ where $\mP\in\rnn$ is a permutation matrix and $\ml\in\cnn$, then
Theorem~\ref{t_eq1} and Remark~\ref{r_eqid1} imply
\begin{align*}
\intdim(\ma)\leq \intdim(\ml).
\end{align*}
\end{remark}

Theorem~\ref{t_eq1} is generalized in Theorem~\ref{t_eq7}.
\section{Extension to p-stable ranks}\label{s_ext}
We review
Schatten $p$-norms (Section~\ref{s_schatten}), and
extend the notion of stable rank and intrinsic dimension to $p$-stable ranks in 
Schatten $p$-norms (Section~\ref{s_pstable}),
thereby unifying the stable rank and intrinsic dimension:
The stable rank is the 2-stable rank, while the intrinsic
dimension is the 1-stable rank of a Hermitian positive
semi-definite matrix.
Subsequently, we  derive 
sum inequalities for the $p$th root of the $p$-stable 
rank (Section~\ref{s_psum}) and
product inequalities 
for the $p$-stable rank (Section~\ref{s_pproduct}),
and show that the $p$th root of the $p$-stable rank
is well-conditioned in the
normwise absolute sense (Section~\ref{s_pcond}).

\subsection{Review of Schatten p-norms}\label{s_schatten} 
The $p$-stable rank is defined in terms of Schatten $p$-norms
which, in turn, are defined on the singular values of real and complex matrices, and thus special cases of symmetric 
gauge functions \cite[Chapter IV]{Bhatia97}, 
\cite[Section 1]{HCH2020}, \cite[Sections 3.4-3.5]{HoJoII}, \cite{RJ2010}.

\begin{definition}\label{d_norms}
Let $\ma\in\cmn$ have singular values $\sigma_1(\ma)\geq \cdots \geq \sigma_r(\ma)$,
where $r\equiv\min\{m,n\}$.
For $1\leq p\leq \infty$, the {\rm Schatten  $p$-norms} are defined as
$$\spn{\ma}\ \equiv \ \sqrt[p]{\sigma_1(\ma)^p+\cdots+\sigma_r(\ma)^p}.$$
\end{definition}
The following are popular Schatten $p$-norms:
\begin{description}
\item[$\qquad p=1:$\  ] Nuclear (trace) norm
$\ \|\ma\|_* \ =\ \sum_{j=1}^r{\sigma_j(\ma)}\ = \son{\ma}$.
\item[$\qquad p=2:$\ ] Frobenius norm
$\ \|\ma\|_F \ =\ \sqrt{\sum_{j=1}^r{\sigma_j(\ma)^2}}\ = \  \stn{\ma}$.
\item[$\qquad p=\infty:$\ ] Two (operator) norm
$\ \|\ma\|_2\ =\ \sigma_1(\ma)\ = \ \sinf{\ma}$.
\end{description}
\bigskip

The Schatten $p$-norms satisfy various properties,
including monotonicity,
submultiplicativity, strong submultiplicativity,
and unitary invariance (symmetric norm), as shown below.

\begin{lemma}\label{l_pnorm}
Let $\ma\in\cmn$ with $r\equiv\rank(\ma)\leq \min\{m,n\}$, then
\begin{align}\label{e_rank}
\spn{\ma}\leq \sqrt[p]{r}\|\ma\|_2.
\end{align}
If $1\leq p\leq q\leq\infty$  then
\begin{align}\label{e_mono}
\son{\ma}\leq\spn{\ma}\leq \spnq{\ma}\leq \sinf{\ma}.
\end{align}
If $\mb\in\complex^{n\times \ell}$, and $\mc\in\complex^{s\times m}$, then
\begin{align}\label{e_multp}
\spn{\ma\mb} &\leq \spn{\ma} \spn{\mb}
\end{align}
and
\begin{align}\label{e_ssm}
\spn{\mc\ma\mb}&\leq \sigma_1(\mc)\,\sigma_1(\mb)\,\spn{\ma} = \|\mc\|_2\,\|\mb\|_2\,\spn{\ma}.
\end{align}
If $\mq_1\in\complex^{s\times m}$ with $\mq_1^T\mq_1=\mi_m$ and
$\mq_2\in\complex^{\ell \times n}$ with $\mq_2^T\mq_2=\mi_n$, then
\begin{align}\label{e_uni}
\spn{\mq_1\ma\mq_2^T} = \spn{\ma}.
\end{align}
\end{lemma}

\subsection{The p-stable rank}\label{s_pstable}
We define the $p$-stable rank and discuss its basic
properties.

\begin{definition}\label{d_pstable}
For $1\leq p\leq \infty$, the {\rm $p$-stable 
rank} of a non-zero matrix $\ma\in\cmn$ is
\begin{align*}
\sr_p(\ma)\equiv \spn{\ma}^p/\sinf{\ma}^p=\spn{\ma}^p/\|\ma\|_2^p.
\end{align*}
The $p$-stable rank of the zero matrix is $\sr_p(\vzero_{m\times n})=0$.
\end{definition}

\Ilse{In analogy to~(\ref{e_ftnorm}),
the relation (\ref{e_rank}) between Schatten
$p$-norms and two-norm can be improved to
\begin{align}
\spn{\ma}&\leq \sqrt[p]{\sr_p(\ma)}\>\sinf{\ma}=
\sqrt[p]{\sr_p(\ma)}\>\|\ma\|_2.
\end{align}
Extending the submultiplicative property~(\ref{e_ftnorm2}) to general Schatten 
$p$-norms gives
\begin{align}\label{e_multp2p}
\spn{\ma\mb} &\leq
\sqrt[p]{\min\{\sr_p(\ma),\sr_p(\mb),\sr_p(\ma\mb)\}}\>
\sinf{\ma}\sinf{\mb}\\ 
&=\sqrt[p]{\min\{\sr_p(\ma),\sr_p(\mb),\sr_p(\ma\mb)\}}\>
\|\ma\|_2\|\mb\|_2. 
\end{align}
Note the difference to (\ref{e_multp}), which bounds a Schatten $p$-norm
in terms of the very same Schatten $p$-norm, rather than the two-norm.
}

\begin{remark}\label{r_unify}
Definition~\ref{d_pstable} unifies the concepts of stable
rank and intrinsic dimension:
The 2-stable rank of a general matrix $\ma\in\cmn$
is the ordinary stable rank,
\begin{align*}
\sr_2(\ma)=\sr(\ma),
\end{align*}
while the 1-stable rank of a 
Hermitian positive semi-definite matrix
$\ma\in\cnn$ is the intrinsic dimension,
\begin{align*}
\sr_1(\ma)=\intdim(\ma).
\end{align*}
For $p=\infty$, we have 
\begin{align*}\sr_{\infty}(\ma)=1.
\end{align*}
\end{remark}

The examples for the stable rank from Example~\ref{ex_sr2}
also hold for the $p$-stable rank.

\begin{example}
The following are special cases where $\sr_p(\ma)=\rank(\ma)$.
\begin{itemize}
\item If $\rank(\ma)=1$ then $\sr_p(\ma)=1$. 
\item If $\ma=\alpha\mU$ for some $\alpha\neq 0$
and unitary, or real orthogonal $\mU\in\cnn$, then $\sr_p(\ma)=n=\rank(\ma)$.
\item If $\ma\in\cmn$ with $\rank(\ma)=r\geq 1$
has singular values $\sigma_1(\ma)=\cdots=\sigma_r(\ma)>0$,
then $\sr_p(\ma)=r=\rank(\ma)$.
\item If $\ma$ is an orthogonal projector, then $\sr_p(\ma) = \spn{\ma}^p = \rank(\ma)$.
\end{itemize}
\end{example}

\begin{remark}
The $p$-stable rank is unitarily invariant. That is, if 
$\ma\in\cmn$, and $\mU\in\cmm$ and $\mv\in\cnn$ are unitary or real orthogonal, then (\ref{e_uni}) implies
\begin{align*}
\sr_p(\ma)=\sr_p(\mU\ma\mv).
\end{align*}
The stable rank is also scale invariant, which implies if $\alpha \in \mathbb{C}$ is nonzero, then 
\[ \sr_p(\alpha\ma) = \sr_p(\ma).\]
\end{remark}

The following generalizes the stable rank relation from 
Section~\ref{s_idsr} between a matrix and its cross product matrix.

\begin{remark}\label{r_idsr}
Let $\ma\in\cmn$. Then
\begin{align*}
\sr_p(\ma^*\ma)=\sr_{2p}(\ma)=\sr_p(\ma\ma^*).
\end{align*}
If $\ma\in\cnn$ is Hermitian positive semi-definite, then
\begin{align*}
\sr_p(\ma)=\sr_{2p}(\ma^{1/2}).
\end{align*}
\end{remark}

The $p$-stable rank inherits the monotonicity 
(\ref{e_mono}) of the Schatten $p$-norms.

\begin{remark}[Relation between $p$-stable ranks]
Let $\ma \in \cmn$ be a nonzero matrix with $\rank(\ma) = r$. If $1 \le q \le p \le \infty$, 
then
\[ 1 \le \sr_{p}(\ma) \le \sr_q(\ma)\le r.\]

This follows from $\mb\equiv \ma/\|\ma\|_2$ having the largest singular equal to 1 and (\ref{e_mono}).
\end{remark}

\begin{remark}[$p$-stable ranks for $0\leq p\leq 1$]
Although the functions in Definition~\ref{d_norms}
are only quasi-norms for $0 \le p < 1$ rather than norms, because they do not satisfy the triangle
inequality \cite[Problem IV.5.1]{Bhatia97}, we can extend Definition~\ref{d_pstable}  to include this range. Indeed, for $p = 0$, we have $\sr_0(\ma) = r$.

Our subsequent results, however, assume that $p\geq 1$.
\end{remark}

\subsection{Sum inequalities for the pth root of the p-stable rank}\label{s_psum}
We show that
the $p$th root of the $p$-stable rank satisfies the
rank-sum inequality~(\ref{e_sum1}), provided both matrices are non-zero Hermitian positive semi-definite.

\begin{theorem}\label{t_sr8}
Let $\ma,\mb\in\cnn$ be non-zero Hermitian positive semi-definite. Then
\begin{align*}
\sqrt[p]{\sr_p(\ma+\mb)}\leq \sqrt[p]{\sr_p(\ma)}+\sqrt[p]{\sr_p(\mb)}.
\end{align*}
\end{theorem}

\begin{proof}
Applying Weyl's monotonicity theorem (\ref{e_weyl}),
$\|\ma+\mb\|_2\geq \|\ma\|_2$ and
$\|\ma+\mb\|_2\geq \|\mb\|_2$,
and the triangle inequality
to the $p$th root of the $p$-stable rank gives
\begin{align*}
\sqrt[p]{\sr_p(\ma+\mb)}\leq\frac{\spn{\ma}}{\|\ma+\mb\|_2}
+\frac{\spn{\mb}}{\|\ma+\mb\|_2}\leq 
\sqrt[p]{\sr_p(\ma)}+\sqrt[p]{\sr_p(\mb)}.
\end{align*}
\end{proof}

In particular,
the square root of the 2-stable rank satisfies the
rank-sum inequality~(\ref{e_sum1}), provided both matrices are non-zero Hermitian positive semi-definite.

\begin{corollary}\label{c_sr8}
Let $\ma,\mb\in\cnn$ be non-zero Hermitian positive semi-definite. Then
\begin{align*}
\sqrt{\sr(\ma+\mb)}\leq \sqrt{\sr(\ma)}+\sqrt{\sr(\mb)}.
\end{align*}
\end{corollary}

Theorem~\ref{thm:intdimsubadditive} shows
that the 1-stable rank satisfies the
rank-sum inequality~(\ref{e_sum1}) for Hermitian positive semi-definite matrices.

\subsubsection{Adding a matrix of rank~1}
The $p$th root of the $p$-stable rank satisfies the second
inequality in (\ref{e_sum3}) for Hermitian
positive semi-definite matrices.
This means, adding a Hermitian positive semi-definite matrix of rank~1 increases the $p$th root of the 
$p$-stable rank by at most one.

\begin{theorem}\label{t_sr9}
Let $\ma,\mb\in\cnn$ be Hermitian positive semi-definite.
If $\rank(\mb)=1$ then
\Ilse{\begin{align*}
1\leq \sqrt[p]{\sr_p(\ma+\mb)}\leq 1+\sqrt[p]{\sr_p(\ma)}.
\end{align*}}
\end{theorem}

\begin{proof}
\Ilse{The proof is similar to that of Theorem~\ref{t_intdim9}.}

\Ilse{Since $\ma$ and $\mb$ are Hermitian positive semidefinite, so is
$\ma+\mb$. In addition, $\rank(\mb)=1$ implies
$\rank(\ma+\mb)\geq \rank(\mb)=1$. Thus $\sr_p(\ma+\mb)\geq 1$.
This establishes the lower bound.}

\Ilse{If $\ma=\vzero$ then $\sr_p(\ma)=0$ and 
$\sr_p(\ma+\mb)=\sr_p(\mb)=1$, so that the upper bound 
holds with equality.}

\Ilse{Now assume that $\ma\neq \vzero$.}
Applying Weyl's monotonicity theorem (\ref{e_weyl}),
$\|\ma+\mb\|_2\geq \|\ma\|_2$ and
$\|\ma+\mb\|_2\geq \|\mb\|_2$
gives
\begin{align*}
\sqrt[p]{\sr_p(\ma+\mb)}-\sqrt[p]{\sr_p(\ma)}&\leq
\frac{\spn{\ma}}{\|\ma+\mb\|_2}+
\frac{\spn{\mb}}{\|\ma+\mb\|_2}-
\frac{\spn{\ma}}{\|\ma\|_2}\\
&\leq\frac{\spn{\ma}}{\|\ma\|_2}+
\frac{\spn{\mb}}{\|\mb\|_2}-\frac{\spn{\ma}}{\|\ma\|_2}\\
&=\sqrt[p]{\sr_p(\mb)}=1.
\end{align*}
The last equality follows from 
the assumption $1=\rank(\mb)=\sr_p(\mb)$.
\end{proof}

\Ilse{If $\ma=\vzero$, then the lower and upper bounds in 
Theorem~\ref{t_sr9} hold with equality. The following example illustrates that 
for $\ma\neq\vzero$ the $p$-stable rank
of the sum in Theorem~\ref{t_sr9} can be arbitrarily close to one.}

\begin{example}
\Ilse{Let 
\begin{align*}
\ma=\begin{bmatrix}0&0\\ 0&1\end{bmatrix},\qquad 
\mb=\begin{bmatrix}\beta&0\\ 0&0\end{bmatrix}, \qquad
\ma+\mb=\begin{bmatrix}1&0\\ 0&\beta\end{bmatrix}
\end{align*}
where $\beta>1$ where $\sr_p(\mb)=1=\rank(\mb)$. Then
\begin{align*}
\sr_p(\ma+\mb)=\frac{1^p+\beta^p}{\beta^p}=1+\frac{1}{\beta^p}
\rightarrow 1\quad \mathrm{as}\ \beta\rightarrow \infty.
\end{align*}
}
\end{example}

As a particular case of Theorem~\ref{t_sr9}, the square root of the 2-stable rank satisfies the second
inequality in (\ref{e_sum3}) for Hermitian
positive semi-definite matrices.

\begin{corollary}
Let $\ma,\mb\in\cnn$ be Hermitian positive semi-definite.
If $\sr(\mb)=1$ then
\Ilse{\begin{align*}
\sqrt{\sr(\ma+\mb)}\leq\sqrt{\sr(\ma)}+ 1.
\end{align*}}
\end{corollary}

Theorem~\ref{t_intdim9} shows that the 1-stable rank satisfies the second
inequality in (\ref{e_sum3}) for Hermitian
positive semi-definite matrices.

\subsection{Product inequalities for the p-stable rank}\label{s_pproduct}
We derive product inequalities for the $p$-stable rank,
as an alternative to the rank inequalities in Section~\ref{s_rankprod}. 

Unfortunately, we were not able to exploit the many
existing 
product inequalities for Schatten $p$-norms
\cite[Section IV.2]{Bhatia97}, \cite[Section 3.5]{HoJoII}
because they tend to hold only for the Hermitian positive
semi-definite polar factor of the product, rather than
the product itself.

The inequalities below, which require one
of the factors to be nonsingular,
bound the $p$-stable
rank of the product by the $p$-stable rank of 
the other factor and the $p$th power of the condition
number of the nonsingular factor.

\begin{theorem}\label{t_srp1}
If $\ma\in\cmm$ is nonsingular, and $\mb\in\cmn$, then
\begin{align*}
\frac{\sr_p(\mb)}{\kappa_2(\ma)^p}\leq 
\sr_p(\ma\mb)\leq \kappa_2(\ma)^p\sr_p(\mb).
\end{align*}
\end{theorem}

\begin{proof}
Applying strong submultiplicativity (\ref{e_ssm})
\begin{align*}
\spn{\ma\mb}\leq \|\ma\|_2\spn{\mb}
\end{align*}
and the singular value product inequalities \cite[7.3.P16]{HoJoI}
to the $p$-stable rank of a product gives the upper bound
\begin{align*}
\sr_p(\ma\mb)\leq \|\ma\|_2^p \frac{\spn{\mb}^p}{\|\ma\mb\|_2^p}
\leq \underbrace{\|\ma\|_2^p\|\ma^{-1}\|_2^p}_{\kappa_2(\ma)^p}\underbrace{\frac{\spn{\mb}^p}{\|\mb\|_2^p}}_{\sr_p(\mb)}.
\end{align*}

For the lower bound, apply the strong submultiplicativity
(\ref{e_ssm})
\begin{align*}
\spn{\mb}=\spn{\ma^{-1}\ma\mb}\leq \|\ma^{-1}\|_2
\spn{\ma\mb}
\end{align*}
to the $p$-stable rank of a product
\begin{align*}
\sr_p(\ma\mb)\geq \frac{1}{\|\ma^{-1}\|_2^p} \frac{\spn{\mb}^p}{\|\ma\mb\|_2^p}
\geq \underbrace{\frac{1}{\|\ma\|_2^p\|\ma^{-1}\|_2^p}}_{1/\kappa_2(\ma)^p}\underbrace{\frac{\spn{\mb}^p}{\|\mb\|_2^p}}_{\sr_p(\mb)}.
\end{align*}
\end{proof}

The bounds in Theorem~\ref{t_srp1} hold with equality 
if $\ma$ is unitary, or real orthogonal.

Below are the special cases of the 2-stable rank, 
and the 1-stable rank for Hermitian positive semi-definite
matrices.

\begin{corollary}\label{c_srp1}
If $\ma\in\cmm$ is nonsingular, and $\mb\in\cmn$, then
\begin{align*}
\frac{\sr(\mb)}{\kappa_2(\ma)^2}\leq 
\sr(\ma\mb)\leq \kappa_2(\ma)^2\sr(\mb).
\end{align*}
If $\ma\in\cnn$ is Hermitian positive definite, and $\mb\in\cnn$ Hermitian positive semi-definite, then
\begin{align*}
\frac{\intdim(\mb)}{\kappa_2(\ma)}\leq
\intdim(\ma\mb)\leq \kappa_2(\ma)\intdim(\mb).
\end{align*}
\end{corollary}
\Ilse{As in Section~\ref{s_rankprod}, for $\intdim(\ma\mb)$ to be defined, 
the product $\ma\mb$
must also be Hermitian positive definite, which is the case when
 $\ma$ and $\mb$ are simultaneously diagonalizable.}

Below is an extension of Theorem~\ref{t_eq1} that shows
that the rank of the cross product matrix is bounded by the
rank of the matrix.

\begin{theorem}\label{t_eq7}
If $\ma\in\cmn$, then
\begin{align*}
\sr_p(\ma^*\ma)\leq \sr_p(\ma)\qquad
\text{and}\qquad \sr_p(\ma\ma^*)\leq \sr_p(\ma).
\end{align*}
\end{theorem}

\begin{proof}
The inequalities follow from 
strong submultiplicativity (\ref{e_ssm}),
$\spn{\ma^*\ma}\leq \|\ma\|_2\spn{\ma}$ and 
$\|\ma^*\ma\|_2=\|\ma\|_2^2=\|\ma\ma^*\|_2$.
\end{proof}

\subsection{Conditioning of the pth root of the p-stable rank}\label{s_pcond}
 Unlike the rank, which is ill-posed, the 
 $p$-stable ranks are well-posed, due to their continuous dependence  on the matrix elements. In particular, the square root of the stable rank and the intrinsic dimension are well-posed. 

 The bounds below show that the $p$th root of the $p$-stable rank is well-conditioned in the normwise absolute sense. The conditioning improves if the
 matrix and the perturbation are Hermitian positive semi-definite.

\begin{theorem}\label{t_cond1}
Let $\ma,\me\in\cmn$ \Ilse{with $\ma\neq \vzero$,}
$r\equiv \rank(\me)$, and
$\epsilon\equiv \frac{\|\me\|_2}{\|\ma\|_2}<1$.
Then
\begin{align*}
 \frac{1}{1+\epsilon}\left(\sqrt[p]{\sr_p(\ma)} 
 -\sqrt[p]{r}\epsilon\right)\leq
\sqrt[p]{\sr_p(\ma+\me)}\leq 
\frac{1}{1-\epsilon}\left(\sqrt[p]{\sr(_p\ma)} +\sqrt[p]{r}\epsilon\right).
\end{align*}
If, in addition, $\ma,\me\in\cnn$ are Hermitian positive semi-definite,
then
\begin{align*}
\frac{\sqrt[p]{\sr_p(\ma)}}{1+\epsilon}\leq 
\sqrt[p]{\sr_p(\ma+\me)}
\leq \sqrt[p]{\sr_p(\ma)}+\sqrt[p]{r}\epsilon.
\end{align*}
\end{theorem}

\begin{proof}
If $\sr_p(\me)=0$, then the bounds clearly hold. Thus, we assume 
$\sr_p(\me)\neq 0$.

\begin{enumerate}
\item First set of bounds.\\
For the lower bound, the triangle inequality and 
(\ref{e_rank}) imply
\begin{align*}
\sqrt[p]{\sr_p(\ma+\me)}&=
\frac{\spn{\ma+\me}}{\|\ma+\me\|_2}
\geq \frac{\spn{\ma}-\spn{\me}}{\|\ma\|_2+\|\me\|_2}\\
&=\frac{\|\ma\|_2}{\|\ma\|_2+\|\me\|_2}
\left(\sqrt[p]{\sr_p(\ma)}
-\frac{\spn{\me}}{\|\ma\|_2}\right)\\
&=\frac{1}{1+\epsilon}\left(\sqrt[p]{\sr_p(\ma)}
-\frac{\spn{\me}}{\|\ma\|_2}\right)\\
&\geq \frac{1}{1+\epsilon}\left(\sqrt[p]{\sr_p(\ma)}
-\sqrt[p]{r}\frac{\|\me\|_2}{\|\ma\|_2}\right)
=\frac{1}{1+\epsilon}\left(\sqrt[p]{\sr_p(\ma)}
-\sqrt[p]{r}\epsilon\right).
\end{align*}
The upper bound is derived analogously.
\item Second set of bounds.\\
In the lower bound, $\me$ being Hermitian positive
semi-definite implies
\begin{align*}
\sqrt[p]{\sr_p(\ma+\me)}&=\frac{\spn{\ma+\me}}{\|\ma+\me\|_2}
\geq\frac{\spn{\ma}}{\|\ma\|_2+\|\me\|_2}
=\frac{\sqrt[p]{\sr_p(\ma)}}{1+\epsilon}.
\end{align*}

Applying Weyl's monotonicity theorem (\ref{e_weyl}),
$\|\ma+\me\|_2\geq\|\ma\|_2$,
to the upper bound gives
\begin{align*}
\sqrt[p]{\sr_p(\ma+\me)}&=\frac{\spn{\ma+\me}}{\|\ma+\me\|_2}
\leq \frac{\spn{\ma+\me}}{\|\ma\|_2}\leq 
\sqrt[p]{\sr_p(\ma)}+\frac{\spn{\me}}{\|\ma\|_2}\\
&\leq \sqrt[p]{\sr(_p\ma)}+\sqrt[p]{r}\epsilon.
\end{align*}
\end{enumerate}
\end{proof}

Below is the special case of the 2-stable rank.

\begin{corollary}\label{c_cond1} 
Let $\ma,\me\in\cmn$ with
$r\equiv \rank(\me)\leq \min\{m,n\}$ and
$\epsilon\equiv \frac{\|\me\|_2}{\|\ma\|_2}<1$. Then
\begin{align*}
 \frac{1}{1+\epsilon}(\sqrt{\sr(\ma)} -\sqrt{r}\epsilon)\le
\sqrt{\sr(\ma+\me)}\leq \frac{1}{1-\epsilon}(\sqrt{\sr(\ma)} +\sqrt{r}\epsilon).
\end{align*}
If $\ma,\me\in\cnn$ are also Hermitian positive semi-definite,
then
\begin{align*}
\frac{\sqrt{\sr(\ma)}}{1+\epsilon}\leq \sqrt{\sr(\ma+\me)}
\leq \sqrt{\sr(\ma)}+\sqrt{r}\epsilon.
\end{align*}
\end{corollary}

Below is the special case of the 1-stable rank for
Hermitian positive semi-definite matrices

\begin{corollary}\label{c_cond2}
Let $\ma\in\cnn$ be Hermitian positive semi-definite and 
$\me\in\cnn$ be Hermitian. If $\ma+\me$ is Hermitian positive
semi-definite and $\epsilon\equiv \frac{\|\me\|_2}{\|\ma\|_2}<1$, then 
\begin{align*}
 \frac{1}{1+\epsilon}(\intdim(\ma) -n\epsilon)\le
\intdim(\ma+\me)\leq \frac{1}{1-\epsilon}(\intdim(\ma) +n\epsilon).
\end{align*}
If, in addition, $\me$ is also positive semi-definite, then 
\begin{align*}
\frac{\intdim(\ma)}{1+\epsilon}\leq \intdim(\ma+\me)\leq \intdim(\ma)+n\epsilon.
\end{align*}
\end{corollary}

\subsection*{Acknowledgement}
We thank Srinivas Eswar for informative discussions, and the reviewers for helpful suggestions. 

\bibliography{rankbib}
\bibliographystyle{siam}

\end{document}